\documentclass[12pt,reqno]{amsart}
\usepackage{amssymb,amsmath,amsthm,enumerate,enumitem,verbatim,bbm,bm,mathrsfs, mathtools, mathrsfs, xcolor}

\usepackage[a4paper]{geometry}
\DeclareMathOperator{\Tr}{Tr}

\DeclareMathOperator{\supp}{supp}

\newcommand{\ess}{\text{\rm ess}}

\DeclareMathOperator{\conv}{conv}

\renewcommand\Re{\hbox{{\rm Re}}\,}

\newcommand{\Abs}[1]{\left\lvert#1\right\rvert}
\newcommand{\norm}[1]{\lVert#1\rVert}

\newcommand{\Norm}[1]{\left\lVert#1\right\rVert}
\newcommand{\Br}[1]{\left(#1\right)}
\newcommand{\CBr}[1]{\left\{#1\right\}}
\newcommand{\SBr}[1]{\left[#1\right]}

\newcommand{\C}{{\mathbb C}}
\newcommand{\D}{{\mathbb D}}

\newcommand{\N}{\mathbb N}
\newcommand{\T}{{\mathbb T}}

\newcommand{\Z}{{\mathbb Z}}

\newcommand{\bS}{{\mathbf S}}

\newcommand{\calE}{\mathcal{E}}

\newcommand{\calR}{\mathcal{R}}

\newcommand{\calX}{\mathcal{X}}

\DeclareFontFamily{U}{mathx}{\hyphenchar\font45}
\DeclareFontShape{U}{mathx}{m}{n}{<5> <6> <7> <8> <9> <10>
<10.95> <12> <14.4> <17.28> <20.74> <24.88> mathx10}{}
\DeclareSymbolFont{mathx}{U}{mathx}{m}{n}
\DeclareFontSubstitution{U}{mathx}{m}{n}
\DeclareMathAccent{\widecheck}{0}{mathx}{"71}

\numberwithin{equation}{section}

\theoremstyle{plain}
\newtheorem{theorem}{\bf Theorem}[section]
\newtheorem{theoremA}{\bf Theorem}

\newtheorem{lemma}[theorem]{\bf Lemma}

\theoremstyle{definition}

\theoremstyle{remark}
\newtheorem*{remark*}{\bf Remark}
\newtheorem{remark}[theorem]{\bf Remark}

\newcommand{\wb}{\overline}
\newcommand{\clos}{\overline}
\newcommand{\wh}{\widehat}

\newcommand{\eps}{\varepsilon}

\newcommand{\1}{1}
\newcommand{\dm}{\,\mathrm{d}}

\renewcommand{\[}{\begin{equation}}
\renewcommand{\]}{\end{equation}}

\begin{document} 
\title{Szeg\H{o} Limit Theorem for Truncated Toeplitz Operators}
\date{\today} 

\author{Nazar Miheisi}
\address{Department of Mathematics,
King's College London,
Strand, London WC2R 2LS,
United Kingdom}
\email{nazar.miheisi@kcl.ac.uk}

\author{Ryan O'Loughlin}
\address{D\'epartement de Math\'ematiques et de Statistique,
Universit\'e Laval,
Qu\'ebec, G1V 0A6,
Canada}
\email{ryan.oloughlin.1@ulaval.ca}



\begin{abstract}
We discuss generalizations of the Szeg\H{o} Limit Theorem to truncated Toeplitz
operators. In particular, we consider compressions of Toeplitz operators to an
increasing sequence of finite dimensional model spaces. We present two theorems.
The first is a new variant of the Szeg\H{o} Limit Theorem in this setting. The
second relates to the variant given by Strouse-Timotin-Zarrabi in \cite{STZ1},
and characterizes the sequences for which that result holds.
\end{abstract}

\maketitle

%%%%%%%%%%%%%%%%%%%%%%%%%%%%%%%%%%%%%%%%%%%%%%%%%%%%%%%%%%%%%%%%%%%%%%%%%%%

\section{Introduction}

Throughout, $\D$ will be the unit disk and $\T=\partial\D$ will be the unit circle.
For $\varphi\in L^\infty(\T)$, the Toeplitz operator $T(\varphi)$ is the compression
of multiplication by $\varphi$ on the Hardy space $H^2=H^2(\D)$. Then the matrix
of $T(\varphi)$ with respect to the standard basis $\{z^j\}_{j\geq0}$ of $H^2$ is
$\{\wh\varphi(j-k)\}_{j,k\geq0}$, where $\wh\varphi(k)$ is the $k^\mathrm{th}$
Fourier coefficient of $\varphi$. One of the principal results in the theory of
Toeplitz operators is the \emph{Szeg\H{o} Limit Theorem}\footnote{This is sometimes
called the First Szeg\H{o} Limit Theorem to distinguish it from the Strong Szeg\H{o}
Limit Theorem which gives a second order asymptotic.} which describes the asymptotic
spectral density of the $N\times N$ truncations
$T_N(\varphi)=\{\wh\varphi(j-k)\}_{j,k=0}^{N-1}$.

\begin{theoremA}[Szeg\H{o} Limit Theorem]\label{thm:SLT-classic}
Let $\varphi\in L^\infty(\T)$ and let $f$ be an analytic function on the closed convex
hull of the essential range of $\varphi$. Then
$$
\lim_{N\to\infty} \frac{1}{N}\Tr f(T_N(\varphi))
=
\int_\T f(\varphi(\zeta))\dm\zeta.
$$
\end{theoremA}
For suitable $\varphi$, we can take $f(x)=\log x$ in Theorem \ref{thm:SLT-classic} to
obtain a version of the Szeg\H{o} Limit Theorem in terms of the Toeplitz determinants
$\det T_N(\varphi)$. One can also go in the other direction and deduce Theorem
\ref{thm:SLT-classic} from the determinant version. Another variant is the Strong
Szeg\H{o} Limit Theorem which, under additional regularity assumptions on $\varphi$,
gives a second order asymptotic expansion for $\Tr f(T_N(\varphi))$. We refer to
\cite[\S 5]{BotSil1} or \cite[\S 5]{Nikolski1} for a more detailed discussion.

Let $B\in H^2$ be an inner function -- that is, $|B|=1$ almost everywhere on $\T$. The
\emph{truncated Toeplitz operator} $T_B(\varphi)$ on the model space $K_B :=
H^2\ominus B H^2$ is the compression of $T(\varphi)$ to $K_B$; we say $\varphi$ is
the \emph{symbol} of $T_B(\varphi)$. Observe that $T_N(\varphi)$ above is precisely
the matrix of $T_B(\varphi)$ with $B(z)=z^N$.

This paper is concerned with generalizations of the Szeg\H{o} Limit Theorem to truncated
Toeplitz operators. In this setting, we consider a sequence $\{\lambda_j\}_{j\geq0}$
in $\D$ and look for an analogue of Theorem \ref{thm:SLT-classic} for the operators
$T_{B_N}(\varphi)$, where $B_N$ is the Blaschke product with zeros $\lambda_0,\dots,
\lambda_{N-1}$. As far as we are aware, two results of this type exist in the literature.
The first is a result of B\"{o}ttcher generalizing the determinant version of the Strong
Szeg\H{o} Limit Theorem \cite{Bottcher1}. The second is a result of Strouse, Timotin and
Zarrabi (Theorem \ref{thm:STZ} below) which gives a trace version of Theorem
\ref{thm:SLT-classic} for continuous symbols, and which is valid when 
$\{\lambda_j\}_{j\geq0}$ is not a Blaschke sequence \cite{STZ1}; for Blaschke sequences
they give only partial results. The aims of this paper are twofold:
\begin{enumerate}[topsep=0pt, itemsep=3pt]
\item
We give an alternative version of the  Szeg\H{o} Limit Theorem which is valid for a
wider class of symbols and all sequences $\{\lambda_j\}_{j\geq0}$;
\item
We characterise those sequences $\{\lambda_j\}_{j\geq0}$ for which the
Strouse-Timotin-Zarrabi theorem holds.
\end{enumerate}
The precise statements are Theorems \ref{thm:SLT} and \ref{thm:angular} below.
The second point addresses a problem raised in \cite{STZ1}.

\subsection{Notation}
Before continuing, we fix some basic notation that is used throughout the
remainder of the paper.
\begin{enumerate}[label=\arabic*.]
\item
The symbol $C$ will always denote a positive constant whose precise value may change
from one occurrence to another.
\item
Normalized Lebesgue measure on $\T$ will be denoted by $m$, although we will simply write
$\dm\zeta$ instead of $\dm m(\zeta)$.
\item
For a function $\varphi$ defined on $\T$, the essential range
of $\varphi$ is denoted by $\calR(\varphi)$ and its closed convex hull by
$\conv\calR(\varphi)$.
\item
The identity function on $\D$ is denoted by $z$, i.e. $z(w)=w$ for all $w\in\D$.
\item
For $\varphi, \psi\in L^2(\T)$, the rank one operator $f\mapsto (f, \psi)\varphi$
is denoted $\varphi\otimes\psi$.
\item
For $p>0$, $\bS_p$ will denote the usual Schatten class of operators on a Hilbert
space. In particular, $\bS_1$ and $\bS_2$ are the trace class and Hilbert-Schmidt
operators respectively.
\end{enumerate}

%%%%%%%%%%%%%%%%%%%%%%%%%%%%%%%%%%%%%%%%%%%%%%%%%%%%%%%%%%%%%%%%%%%%%%%%%%%

\subsection*{Acknowledgement}

We are grateful to Yaqub Alwan for  useful discussions with the first author.

%%%%%%%%%%%%%%%%%%%%%%%%%%%%%%%%%%%%%%%%%%%%%%%%%%%%%%%%%%%%%%%%%%%%%%%%%%%

\section{Main results}

In the remainder of the paper, we let $\Lambda = \{\lambda_j\}_{j\geq0}$ be
a sequence in $\D$. For simplicity, we will suppose that $\lambda_0=0$. For
$N\in\N$, we set
$$
B_N(\zeta)=\prod_{j=0}^{N-1} \frac{\wb{\lambda_j}}{|\lambda_j|}
\frac{\zeta-\lambda_j}{1-\wb{\lambda_j}z},
$$
with the convention that $\wb{\lambda_j}/|\lambda_j|=1$ when $\lambda_j=0$.
In addition, we define the set
$$
\calE =
\CBr{\zeta\in\T :\, \lim_{N\to\infty} |B_N'(\zeta)|<\infty}.
$$
If $\Lambda$ is a Blaschke sequence and $B$ is the corresponding Blaschke product,
this is precisely the set of points at which $B$ has a unimodular bounadry value
and finite angular derivative. If $\Lambda$ is not a Blaschke sequence, $\calE$
will be empty.

\subsection{Szeg\H{o} Limit Theorem for truncated Toeplitz operators}

Let $\nu_N$ be the mean of the harmonic measures on $\T$ corresponding to the points
$\lambda_0,\dots,\lambda_{N-1}$, i.e.
$$
\dm\nu_N(\zeta) = \frac{1}{N}|B_N'(\zeta)|\dm \zeta
= \frac{1}{N}\sum_{j=0}^{N-1}\frac{1-|\lambda_j|^2}{|\zeta-\lambda_j|^2}\dm\zeta.
$$
Then the space $\calX_\Lambda$ is defined to be the closure of $C(\T)$ in the
norm
$$
\norm{f}_\Lambda = \sup_{N\geq1} \norm{f}_{L^2(\nu_N)}.
$$
For example, if $|\lambda_j|\leq r<1$ for all $j$, then $\calX_\Lambda = L^2(\T)$.
More generally, if $\Lambda$ does not accumulate at every point of $\T$,
then $\calX_\Lambda$ contains all $f$ in $L^2(\T)$ that are continuous on a
neighbourhood of $\clos{\Lambda}\cap\T$. We can now state our first result.

\begin{theorem}\label{thm:SLT}
Let $\varphi\in L^\infty(\T)\cap\calX_\Lambda$ and let $f$ be an analytic
function on a neighbourhood of $\conv\calR(\varphi)$. Then
\[
\lim_{N\to\infty}\Abs{\frac{1}{N}\Tr f(T_{B_N}(\varphi))
- \int_\T f(\varphi(\zeta))\dm\nu_N(\zeta)}
=0.
\label{eq:SLT}
\]
Moreover, if $\varphi$ is real-valued, \eqref{eq:SLT} holds for any $f$ which is
continuous on $[\ess\inf\varphi,\ess\sup\varphi]$.
\end{theorem}

\begin{remark}
Unlike in the classical case, the sequence $N^{-1}\Tr f(T_{B_N}(\varphi))$ need not
converge. For example, fix $\lambda\in(0,1)$ and define $\lambda_j\in\D$ as follows:
$$
\lambda_0=0,
\quad
\lambda_j = \begin{cases}
\lambda, \quad 3^{k-1}< j \leq 3^k \;\text{with}\; k \;\text{odd}, \\
-\lambda, \quad 3^{k-1}< j \leq 3^k \;\text{with}\; k \;\text{even}.
\end{cases}
$$
Then for $f=\varphi=z$, an easy calculation shows that for $N=3^k$
$$
\int_\T f(\varphi(\zeta))\dm\nu_N(\zeta) \geq \lambda/3 \quad \text{when $k$ is odd},
$$
and
$$
\int_\T f(\varphi(\zeta))\dm\nu_N(\zeta) \leq -\lambda/3 \quad \text{when $k$ is even}.
$$
However, if $\nu_N$ converges in the weak-$*$ topology to some measure $\nu$, then we
will have that 
\[
\lim_{N\to\infty}\frac{1}{N}\Tr f(T_{B_N}(\varphi))
=
\int_\T f(\varphi(\zeta))\dm\nu(\zeta)
\label{eq:SLT-convergence}
\]
for $\varphi$ and $f$ as in the statement of Theorem \ref{thm:SLT}. Observe that integration
with respect to $\nu$ necessarily extends to a bounded functional on $\calX_\Lambda$,
and since $\calX_\Lambda$ is stable under analytic compositions, the right
hand side of \eqref{eq:SLT-convergence} is well-defined.
\end{remark}

\subsection{The Theorem of Strouse-Timotin-Zarrabi}

Let us introduce the functions
$$
\beta_N(\zeta)=\frac{1}{|B_N'(\zeta)|}.
$$
Strouse, Timotin and Zarrabi proved the following extension of the Szeg\H{o}
Limit Theorem \cite{STZ1}:

\begin{theoremA}[Strouse-Timotin-Zarrabi]\label{thm:STZ}
Assume $\sum (1-|\lambda_j|) = \infty$. Let $\varphi\in C(\T)$ and let $f$ be an
analytic function on a neighbourhood of $\conv\calR(\varphi)$. Then
\[
\lim_{N\to\infty}\Tr \SBr{T_{B_N}(\beta_N)f(T_{B_N}(\varphi))}
=
\int_\T f(\varphi(\zeta))\dm \zeta.
\label{eq:STZ}
\]
Moreover, if $\varphi$ is real-valued \eqref{eq:STZ} holds for any $f$ which is
continuous on $[\inf\varphi, \sup\varphi]$.
\end{theoremA}

The case of a Blaschke sequence is more subtle. This is because the proof relies on
the condition
\[
\norm{\beta_N(U_\alpha^{B_N})}\to0
%\quad\text{as}\quad N\to\infty
\quad\text{for almost every}\; \alpha\in\T,
\label{eq:STZ-condition}
\]
where $\{(U_\alpha^{B_N}): \alpha\in\T\}$ are the Clark unitary operators associated
to $B_N$ (a definition is given in \S \ref{subsec:Clark}). This condition is
automatically satisfied for non-Blaschke sequences, but might not hold for a Blaschke
sequence. In fact, Strouse, Timotin and Zarrabi proved that if \eqref{eq:STZ-condition}
is satisfied, the conclusions of Theorem \ref{thm:STZ} remain valid for $\varphi$ in
the uniform closure of
$$
\bigcup_{N\geq1} \Br{K_{B_N}+\wb{K_{B_N}}}.
$$
However, they also gave an example of a Blaschke sequence for which
\eqref{eq:STZ-condition} is not satisfied and the conclusion of the theorem fails.
They left open the problem of characterizing sequences for which \eqref{eq:STZ-condition}
holds; our second result gives such a characterization.

\begin{theorem}\label{thm:angular}
The following are equivalent:
\begin{enumerate}[label=(\alph*), itemsep=5pt]
\item\label{angular-1}
For almost every $\alpha\in\T$, 
$\norm{\beta_N(U_\alpha^{B_N})}\to0$ as $N\to\infty$;
\item\label{angular-2}
The Lebesgue measure of $\calE$ is zero, i.e.
$$
\sum_{j=1}^\infty \frac{1-|\lambda_j|^2}{|\zeta-\lambda_j|^2} =\infty
\quad\text{for a.e.}\;\;\zeta\in\T.
%\label{eq:angular}
$$
\item\label{angular-3}
For every $\varphi\in C(\T)$ and every $f$ which is analytic on a neighbourhood
of $\conv\calR(\varphi)$, \eqref{eq:STZ} holds.
\end{enumerate}
\end{theorem}

\begin{remark}
\begin{enumerate}[leftmargin=*]
\item
Theorem \ref{thm:angular} extends the result in 
\cite[Theorem 7.1]{STZ1} by showing that the conclusion holds for all symbols
$\varphi\in C(\T)$. 
\item
It's clear that if $\Lambda$ is a Blaschke sequence, then a necessary requirement
for the conditions in Theorem \ref{thm:angular} to hold is that $\Lambda$
accumulates at all points of $\T$. However, this is far from sufficient. For
example, a theorem of Frostman asserts that if
$$
\sum_{j=1}^\infty \sqrt{1-|\lambda_j|} < \infty,
$$
then $|B'|<\infty$ almost everywhere on $\T$ (see, e.g., \cite[Theorem 7.17]{Mashregi1}).
\end{enumerate}
\end{remark}

%%%%%%%%%%%%%%%%%%%%%%%%%%%%%%%%%%%%%%%%%%%%%%%%%%%%%%%%%%%%%%%%%%%%%%%%%%%

\section{Preliminaries}

This section is devoted to summarising some necessary background material. For a
more comprehensive exposition, good sources are \cite[\S 5]{GMR1} for the material on
model spaces, and \cite[\S 11]{GMR1} or \cite{Matheson-Stessin1, Saksman1} for the material
on Clark unitary operators and Clark measures.

\subsection{Model spaces}

Our setting is the usual Hardy space $H^2=H^2(\D)$ of analytic functions on $\D$ with
square summable Taylor coefficients, considered as a subspace of $L^2(\T)$ by means of
boundary values. This is a reproducing kernel Hilbert space and the reproducing
kernel at $\lambda\in\D$ is $k_\lambda(w) = (1-\wb\lambda w)^{-1}$. We let $\wh{k}_\lambda
=\sqrt{1-|\lambda|^2}\,k_\lambda$ be the $L^2$-normalised kernel.

Let $B$ be an inner function. The corresponding \emph{model space} is $K_B :=
H^2\ominus B H^2$. It is standard that $K_B$ is finite dimensional if and only if $B$ is
a finite Blashke product; in this case the dimension of $K_B$ is the degree of $B$.
We say that $B$ has finite angular derivative at $\zeta\in\T$ if $B$ and $B'$ have
non-tangential limits at $\zeta$ and $|B(\zeta)|=1$. If $B$ is a Blaschke product with
zeros $\{\lambda_j\}_{j\geq0}$, this is equivalent to the condition
$$
|B'(\zeta)|
=
\sum_{j\geq0} |\wh{k}_{\lambda_j}(\zeta)|^2
<
\infty.
$$
Then if either $\lambda\in\D$ or $\lambda\in\T$ and $B$ has finite angular derivative
at $\lambda$, point evaluation at $\lambda$ is bounded on $K_B$ and the corresponding
reproducing kernel is
$$
k_\lambda^B(w) = \frac{1-\wb{B(\lambda)}B(w)}{1-\wb\lambda w}.
$$
As before, we write $\wh{k}_\lambda^B$ for the $L^2$-normalized kernel. Note that if
$\lambda\in\T$, then $\|k_\lambda\|_2 = \sqrt{|B'(\lambda)|}$. Moreover, if
$B$ is a finite Blaschke product, then for each $\alpha\in\T$, the set
$\{\wh{k}_\zeta^B : B(\zeta)=\alpha\}$ is an orthonormal basis for $K_B$.

Another natural choice of orthonormal basis for $K_B$ when $B$ is a Blaschke
product is the so-called Takenaka–Malmquist–Walsh basis. This is given by
$\{B_j \wh{k}_{\lambda_j}\}_{j\geq0}$, where $\{\lambda_j\}_{j\geq0}$ are the
zeros of $B$, $B_0=1$, and for $j\geq1$, $B_j$ is the Blaschke product with
zeros $\lambda_0, \dots, \lambda_{j-1}$. Using this, we see that if $\varphi
\in L^\infty(\T)$ is such that $T_B(\varphi)\in\bS_1$, then
\[
\Tr(T_B(\varphi))
= \sum_{j\geq0} \Br{\varphi, |\wh{k}_{\lambda_j}|^2}
= \int_\T \varphi(\zeta) |B'(\zeta)|\dm \zeta.
\label{eq:trace-formula}
\]

\subsection{Clark unitary operators and Clark measures}\label{subsec:Clark}
 
Let $B$ be an inner function with $B(0)=0$. Recall that $T_B(\varphi)$ is the truncated
Toeplitz operator on $K_B$ with symbol $\varphi$. The operator $T_B(z)$ plays a special
role and is called the \emph{compressed shift}. For each $\alpha\in\T$, the rank one
perturbation of $T_B(z)$ given by
$$
U^B_\alpha = T_B(z) + \1\otimes \wb{z}B
$$
is unitary and cyclic. The collection $\{U^B_\alpha: \alpha\in\T\}$ are known as the
\emph{Clark unitary operators} associated with $B$.

By the Spectral Theorem, each $U^B_\alpha$ is unitarily equivalent to multiplication
by $z$ on $L^2(\mu^B_\alpha)$ for some positive measure $\mu^B_\alpha$ on $\T$.
The spectral measures  $\{\mu^B_\alpha: \alpha\in\T\}$ are called the \emph{Clark measures}
associated with $B$. It is known that the set $B^{-1}(\alpha)=\left\{\zeta\in\T:
\lim_{r\to1^-}B(r\zeta)=\alpha\right\}$ is a carrier for $\mu^B_\alpha$, in the sense
that $\mu^B_\alpha(E)=\mu^B_\alpha(E\cap B^{-1}(\alpha))$ for every Borel set
$E\subseteq\T$. If $\zeta\in B^{-1}(\alpha)$, then $\mu^B_\alpha$ has have an
atom at $\zeta$ precisely when $|B'(\zeta)|<\infty$. In this case,
$$
\mu^B_\alpha(\{\zeta\})=\frac{1}{|B'(\zeta)|},
$$
and $k^B_\zeta$ is an eigenvector of $U^B_\alpha$ with eigenvalue $\zeta$.
Consequently, if $B$ is a finite Blaschke product and $\varphi$ is a bounded Borel
function on $\T$, then
\[
\varphi(U_\alpha^B)
=
\sum_{\zeta\in B^{-1}(\alpha)} \frac{\varphi(\zeta)}{|B'(\zeta)|}
k_\zeta^B\otimes k_\zeta^B
=
\int_\T \varphi(\zeta) k_\zeta^B\otimes k_\zeta^B \dm\mu_\alpha^B (\zeta),
\label{eq:Clark-rep}
\]
and hence for each $p\geq0$,
\[
\Norm{\varphi(U_\alpha^B)}_{\bS_p}^p
=
\int_\T |\varphi(\zeta)|^p |B'(\zeta)| \dm\mu_\alpha^B.
\label{eq:Clark-Sp}
\]

A deep property of Clark measures is the Aleksandrov Disintegration Theorem which states that if $f\in L^1(\T)$, then $f\in L^1(\mu^B_\alpha)$
for almost every $\alpha\in\T$ and
$$
\int_\T \int_\T f(\zeta) \dm\mu^B_\alpha(\zeta) \dm\alpha
=
\int_\T f(\zeta) \dm\zeta.
$$
One application of this is the following operator disintegration formula for $T_B(\varphi)$
\cite[Theorem 13.19]{GMR1}:
\[
T_B(\varphi) = \int_\T \varphi(U^B_\alpha) \dm\alpha.
\label{eq:op-disintegration}
\]

%%%%%%%%%%%%%%%%%%%%%%%%%%%%%%%%%%%%%%%%%%%%%%%%%%%%%%%%%%%%%%%%%%%%%%%%%%%

\section{Proofs}

\subsection{Fej\'er-type operators}
The proofs of Theorems \ref{thm:SLT} and \ref{thm:angular} rely on properties
of a sequence of averaging operators associated to $\Lambda$. In the case that
$\lambda_j=0$ for all $j\geq0$, these are precisely the Fej\'er operators (i.e.
convolution with a Fej\'er kernel). We begin by introducing these.

For each $N\in\N$ and $f\in L^2(\T)$, we define the function $E_Nf:\T\to\C$ by
$$
E_Nf(\zeta) = \int_\T f(\eta) |\wh{k}^{B_N}_\zeta(\eta)|^2\dm\eta.
$$
The next lemma summarises the properties of $E_N$ that we need.

\begin{lemma}\label{lem:fejer}
The following hold:
\begin{enumerate}[label=(\alph*), itemsep=5pt]
\item\label{lem:fejer-1}
$E_N$ is a contraction from $L^2(\nu_N)$ into $L^2(\nu_N)$.
\item\label{lem:fejer-3}
If $\zeta\in\T\setminus\calE$ and $f$ is continuous on a neighbourhood of
$\zeta$, then $E_Nf(\zeta)\to f(\zeta)$ as $N\to\infty$.
\item\label{lem:fejer-4}
If $f\in\calX_\Lambda$, then $\Norm{E_N f-f}_{L^2(\nu_N)}\to0$ as $N\to\infty$.
\end{enumerate}
\end{lemma}

\begin{proof}
\ref{lem:fejer-1}
For $f\in L^2(\nu_N)$ we have
\begin{multline}
\Norm{E_Nf}_{L^2(\nu_N)}^2
=
\int_\T\Abs{\int_\T f(\eta)
|\wh{k}^{B_N}_\zeta(\eta)|^2\dm\eta}^2\dm\nu_N(\zeta) \\
\leq
\int_\T \int_\T |f(\eta)|^2
|\wh{k}^{B_N}_\zeta(\eta)|^2\frac{|B_N'(\zeta)|}{N}\dm\eta\dm\zeta.
\label{eq:fejer-1}
\end{multline}
Note that
$$
|\wh{k}^{B_N}_\zeta(\eta)|^2
=
\frac{|B_N'(\eta)|}{|B_N'(\zeta)|} |\wh{k}^{B_N}_\eta(\zeta)|^2,
$$
and so the right hand side of \eqref{eq:fejer-1} is equal to
$$
\int_\T \int_\T |f(\eta)|^2
|\wh{k}^{B_N}_\eta(\zeta)|^2\frac{|B_N'(\zeta)|}{N}\dm\eta\dm\zeta
=
\int_\T |f(\eta)|^2 \int_\T
|\wh{k}^{B_N}_\eta(\zeta)|^2\dm\zeta\dm\nu_N(\eta)
=
\norm{f}_{L^2(\nu_N)}^2.
$$

\ref{lem:fejer-3}
Take $\zeta\in\T\setminus\calE$ and $f\in L^2(\T)$ such that $f$ is continuous
on a neighbourhood of $\zeta$. Fix $\eps>0$ and Choose $\delta>0$ such that
$|f(\zeta)-f(\eta)|<\eps$ whenever $\eta\in I_\delta(\zeta)$, where
$$
I_\delta(\zeta)=\{\eta\in\T: |\arg\,\wb{\eta}\zeta|<\delta\}
$$
is the arc of $\T$ of width $2\delta$ centred at $\zeta$. Then
$$
f(\zeta) - E_N f(\zeta)
=
\int_\T (f(\zeta)-f(\eta))|\wh{k}_\zeta^{B_N} (\eta)|^2 \dm\eta.
$$
Observe that if $\eta\in\T\setminus I_\delta(\zeta)$, then
$$
|\wh{k}_\zeta^{B_N} (\eta)|^2
=
\frac{1}{|B_N'(\zeta)|}\Abs{\frac{B_N(\zeta)-B_N(\eta)}{\zeta-\eta}}^2
\leq
\frac{2}{\delta^2 |B_N'(\zeta)|}.
$$
Then we see that
$$
\Abs{\int_{I_\delta(\zeta)} (f(\zeta)-f(\eta))|\wh{k}_\zeta^{B_N} (\eta)|^2 \dm\eta}
\leq
\eps\int_{I_\delta(\zeta)} |\wh{k}_\zeta^{B_N} (\eta)|^2 \dm\eta
\leq \eps
$$
and
$$
\Abs{\int_{\T\setminus I_\delta(\zeta)} (f(\zeta)-f(\eta))
|\wh{k}_\zeta^{B_N} (\eta)|^2 \dm\eta}
\leq
\frac{4\norm{f}_{\infty}}{\delta^2|B_N'(\zeta)|}.
$$
Letting $N\to\infty$ we conclude that $|f(\zeta) - E_N f(\zeta)|<\eps$.

\ref{lem:fejer-4}
Let $f$ be a Lipschitz continuous function. Then
\begin{align*}
\Norm{E_Nf-f}_{L^2(\nu_N)}^2
&=
\int_\T\Abs{\int_\T \Br{f(\eta)-f(\zeta)}
|\wh{k}^{B_N}_\zeta(\eta)|^2\dm\eta}^2\dm\nu_N(\zeta) \\
&\leq
\frac{1}{N}\int_\T \int_\T |f(\eta)-f(\zeta)|^2
|\wh{k}^{B_N}_\zeta(\eta)|^2|B_N'(\zeta)|\dm\eta\dm\zeta \\
&\leq
\frac{1}{N}\int_\T \int_\T
\frac{|f(\eta)-f(\zeta)|^2}{|\eta - \zeta|^2}
|B_N(\eta)-B_N(\zeta)|^2\dm\eta\dm\zeta \\
&\leq C\frac{1}{N},
\end{align*}
which verifies the claim when $f$ is Lipschitz. For a general
$f\in\calX_\Lambda$ and $\eps>0$, choose a Lipschitz function $g$ such that
$\norm{f-g}_\Lambda <\eps$. Then using part \ref{lem:fejer-1}, we see that
for all sufficiently large $N$,
$$
\Norm{E_Nf-f}_{L^2(\nu_N)}
=
\Norm{E_N(f-g)}_{L^2(\nu_N)}
+
\Norm{E_Ng-g}_{L^2(\nu_N)}
+
\Norm{g-f}_{L^2(\nu_N)}
<
3\eps.
$$
\end{proof}

\subsection{Proof of Theorem \ref{thm:SLT}}

The proof of Theorem \ref{thm:SLT} is based on the well-established approach
of ``approximating by circulants'' (see e.g. \cite[\S 5.6.4]{Nikolski1} or
\cite[\S 4]{Gray1}). As in \cite{STZ1}, the role of circulants will be played by
functions of Clark unitary operators. Our approximation result is the content
of the following lemma.

\begin{lemma}\label{lem:approx}
For each $\varphi\in L^\infty(\T)\cap\calX_\Lambda$,
$$
\lim_{N\to\infty}\frac{1}{N}\int_\T 
\Norm{T_{B_N}(\varphi) - \varphi\Br{U_\alpha^{B_N}}}_{\bS_2}^2 \dm\alpha
=0.
$$
\end{lemma}

\begin{proof}
Observe that we can express $T_{B_N}(\varphi)$ as an integral of the rank one
operators $k^{B_N}_\zeta\otimes k^{B_N}_\zeta$. Indeed, by integrating
\eqref{eq:Clark-rep} with respect to $\alpha$ we see that
$$
T_{B_N}(\varphi) = \int_\T \varphi(\zeta)\,
k^{B_N}_\zeta\otimes k^{B_N}_\zeta \dm\zeta.
$$
Then using the identity $\Tr\SBr{(f\otimes f)(g\otimes g)} = |(f,g)|^2$,
we compute that
\begin{align*}
\Norm{T_{B_N}(\varphi)}_{\bS_2}^2
&=
\Tr\Br{T_{B_N}(\varphi)^* T_{B_N}(\varphi)} \\
&=
\int_\T\int_\T \wb{\varphi(\zeta)}\varphi(\eta)
\Tr\SBr{(k^{B_N}_\zeta\otimes k^{B_N}_\zeta)(k^{B_N}_\eta\otimes k^{B_N}_\eta)}
\dm\zeta \dm\eta \\
&=
\int_\T\int_\T \wb{\varphi(\zeta)}\varphi(\eta)
\Abs{\Br{k^{B_N}_\zeta, k^{B_N}_\eta}}^2 \dm\zeta \dm\eta \\
&=
\int_\T\int_\T \wb{\varphi(\zeta)}\varphi(\eta)
\Abs{\wh{k}^{B_N}_\zeta(\eta)}^2|B_N'(\zeta)| \dm\zeta \dm\eta \\
&=
\int_\T \wb{\varphi(\zeta)}E_N\varphi(\zeta) |B_N'(\zeta)| \dm\zeta.
\end{align*}
A similar computation gives
$$
\Tr\Br{\varphi\Br{U_\alpha^{B_N}}^*T_{B_N}(\varphi)}
=
\int_\T \wb{\varphi(\zeta)}E_N\varphi(\zeta) |B_N'(\zeta)|
\dm\mu_\alpha^N(\zeta).
$$
Using these and \eqref{eq:Clark-Sp}, we see that
\begin{align*}
\Norm{T_{B_N}(\varphi) - \varphi\Br{U_\alpha^{B_N}}}_{\bS_2}^2
&=\
\Norm{T_{B_N}(\varphi)}_{\bS_2}^2 + \Norm{ \varphi\Br{U_\alpha^{B_N}}}_{\bS_2}^2
-2\Re\Tr\Br{\varphi\Br{U_\alpha^{B_N}}^*T_{B_N}(\varphi)} \\
&=
\int_\T \wb{\varphi(\zeta)}E_N\varphi(\zeta) |B_N'(\zeta)| \dm\zeta
+ \int_\T |\varphi(\zeta)|^2|B_N'(\zeta)|\dm\mu_\alpha^N(\zeta) \\
&\qquad\qquad\qquad
 - 2\Re \int_\T \wb{\varphi(\zeta)}E_N\varphi(\zeta) |B_N'(\zeta)|
\dm\mu_\alpha^N(\zeta).
\end{align*}
Integrating with respect to $\alpha$ and applying Cauchy-Schwartz, we find
\begin{align*}
\frac{1}{N}\int_\T
\Norm{T_{B_N}(\varphi) - \varphi\Br{U_\alpha^{B_N}}}_{\bS_2}^2 \dm\alpha
&=
\int_\T \wb{\varphi(\zeta)}\Br{\varphi(\zeta)- E_N\varphi(\zeta)}\dm\nu_N(\zeta) \\
&\leq
\norm{\varphi}_{\Lambda}\norm{\varphi-E_N\varphi}_{L^2(\nu_N)},
\end{align*}
which goes to zero by Lemma \ref{lem:fejer}\ref{lem:fejer-4}.
\end{proof}

\begin{proof}[Proof of Theorem \ref{thm:SLT}.]
Fix $\varphi\in L^\infty(\T)\cap\calX_\Lambda$ and a function $f$ which is analytic on a
neighbourhood of $\conv\calR(\varphi)$. Let $\Omega$ be an open neighbourhood
of $\conv\calR(\varphi)$ with smooth boundary $\partial\Omega$ and such that
$f$ is analytic on a neighbourhood of $\wb{\Omega}$. Then
\begin{multline*}
f(T_{B_N}(\varphi))-f(\varphi(U_\alpha^{B_N}))
=
\int_{\partial\Omega} f(\lambda) [(\lambda I- T_{B_N}(\varphi))^{-1}
- (\lambda I- \varphi(U_\alpha^{B_N}))^{-1}]\dm\lambda \\
=
\int_{\partial\Omega} f(\lambda) (\lambda I- T_{B_N}(\varphi))^{-1}
(T_{B_N}(\varphi)-\varphi(U_\alpha^{B_N}))
(\lambda I-\varphi(U_\alpha^{B_N}))^{-1}\dm\lambda.
\end{multline*}
It immediately follows that
$$
\Norm{f(T_{B_N}(\varphi))-f(\varphi(U_\alpha^{B_N}))}_{\bS_2}
\leq C
\Norm{T_{B_N}(\varphi)-\varphi(U_\alpha^{B_N})}_{\bS_2}.
$$
As a consequence, we see that
\begin{align*}
\frac{1}{N}\Norm{f(T_{B_N}(\varphi))-T_{B_N}(f\circ\varphi)}_{\bS_2}^2
&=
\frac{1}{N}\Norm{\int_\T
\SBr{f(T_{B_N}(\varphi))-f(\varphi(U_\alpha^{B_N}))}\dm\alpha}_{\bS_2}^2 \\
&\leq
\frac{1}{N}\int_\T \Norm{f(T_{B_N}(\varphi))-
f(\varphi(U_\alpha^{B_N}))}_{\bS_2}^2\dm\alpha \\
&\leq
C\frac{1}{N}\int_\T \Norm{T_{B_N}(\varphi) - \varphi(U_\alpha^{B_N})}_{\bS_2}^2 \dm\alpha,
\end{align*}
which goes to zero by Lemma \ref{lem:approx}.

Using the formula \eqref{eq:trace-formula} we have that
$$
\frac{1}{N}\Tr T_{B_N}(f\circ\varphi) = \int_\T f(\varphi(\zeta))\dm\nu_N(\zeta),
$$
and hence
\begin{align*}
\Abs{\frac{1}{N}\Tr f(T_{B_N}(\varphi))
- \int_\T f(\varphi(\zeta))\dm\nu_N(\zeta)}^2 
&=
\Abs{\frac{1}{N}\Tr\SBr{f(T_{B_N}(\varphi))-T_{B_N}(f\circ\varphi)}}^2 \\
&\leq
\frac{1}{N}\Norm{f(T_{B_N}(\varphi))-T_{B_N}(f\circ\varphi)}_{\bS_2}^2,
\end{align*}
which tends to zero. This completes the proof when $f$ is analytic.

When $\varphi$ is real-valued and $f$ is continuous we use a standard
approximation argument. Fix $\eps>0$ and choose a polynomial $p$ such that
$|f(t)-p(t)|<\eps$ for all $t\in\conv\calR(\varphi)$. Then
if $\lambda_1\dots\lambda_N$ are the eigenvalues of $T_{B_N}(\varphi)$,
\begin{align*}
\frac{1}{N}\Abs{\Tr f(T_{B_N}(\varphi)) - \Tr p(T_{B_N}(\varphi))}
&=
\frac{1}{N}\Abs{\sum_{j=1}^N (f(\lambda_j)-p(\lambda_j))} \\
&\leq
\frac{1}{N}\sum_{j=1}^N \Abs{f(\lambda_j)-p(\lambda_j)}  < \eps.
\end{align*}
In addition,
$$
\Abs{\int_\T f(\varphi(\zeta))\dm\nu_N(\zeta)
- \int_\T p(\varphi(\zeta))\dm\nu_N(\zeta)}
\leq
\int_\T |f(\varphi(\zeta))-p(\varphi(\zeta))|\dm\nu_N(t) <\eps.
$$
We conclude that
\begin{multline*}
\Abs{\frac{1}{N}\Tr f(T_{B_N}(\varphi))
- \int_\T f(\varphi(\zeta))\dm\nu_N(\zeta)} \\
\leq
\Abs{\frac{1}{N}\Tr p(T_{B_N}(\varphi))
- \int_\T p(\varphi(\zeta))\dm\nu_N(\zeta)} +2\eps.
\end{multline*}
The first term on the right tends to zero and so the proof is complete.
\end{proof}

%%%%%%%%%%%%%%%%%%%%%%%%%%%%%%%%%%%%%%%%%%%%%%%%%%%%%%%%%%%%%%%%%%%%%%%%%%%

\subsection{Proof of Theorem \ref{thm:angular}}

We begin with two preliminary lemmas.

\begin{lemma}\label{lem:approx-product}
Let $\varphi$ and  $\psi$ be trigonometric polynomials. Then there exists $C>0$ such
that for all inner functions $B$ we have
$$
\Norm{T_B(\varphi)T_B(\psi) - T_B(\varphi\psi)}_{\bS_1} \leq C.
$$
\end{lemma}

\begin{proof}
We can suppose without loss of generality that $B(0)=0$. We begin by writing
$$
T_B(\varphi)T_B(\psi) - T_B(\varphi\psi)
=
\sum_{j,k=-\infty}^\infty
\wh\varphi(j)\wh\psi(k) \left[T_B(z^j)T_B(z^k)-T_B(z^{j+k})\right].
$$
Note that since $\varphi$ and $\psi$ are trigonometric polynomials, the sum on
the right is finite. In addition, if $j,k\in\Z$ have the same sign, 
$T_B(z^j)T_B(z^k)-T_B(z^{j+k})=0$. Thus, by interchanging $j$ and $k$ or taking adjoints
if necessary, it suffices to prove the claim for $\varphi=z^j$ and $\psi=\wb{z}^k$ with
$0\leq k\leq j$. In this case we have
\begin{align*}
T_B(z^j)T_B(\wb{z}^k)-T_B(z^{j-k})
&=
T_B(z^{j-k})\SBr{T_B(z^k)T_B(\wb{z}^k)-I} \\
&=
T_B(z^{j-k})\sum_{m=0}^{k-1}T_B(z^m)\SBr{T_B(z)T_B(\wb{z})-I}T_B(\wb{z}^m).
\end{align*}
A simple calculation shows that $I-T_B(z)T_B(\wb{z})=\1\otimes\1$, and so
$$
\Norm{T_B(z^j)T_B(\wb{z}^k)-T_B(z^{j-k})}_{\bS_1}
\leq
\sum_{m=0}^{k-1}\Norm{T_B(z)T_B(\wb{z})-I}_{\bS_1}
\leq
C.
$$
\end{proof}

\begin{lemma}\label{lem:approx-STZ}
Assume that $\Norm{T_{B_N}(\beta_N)}\to0$ as $N\to\infty$. Then for each
$\varphi\in C(\T)$ and each function $f$ which is analytic on a neighbourhood
of $\conv\calR(\varphi)$, we have
\[
\lim_{N\to\infty}
\Norm{T_{B_N}(\beta_N)\SBr{f(T_{B_N}(\varphi)) - T_{B_N}(f\circ\varphi)}}_{\bS_1} = 0
\label{eq:approx-STZ}
\]
\end{lemma}

\begin{proof}
Let us first show that if $\varphi$ is a trigonometric polynomial, then for
each integer $k\geq0$,
\[
\sup_{N\geq0}\Norm{T_{B_N}(\varphi)^k - T_{B_N}(\varphi^k)}_{\bS_1} < \infty.
\label{eq:approx-STZ2}
\]
We use induction on $k$. In particular, since
\begin{multline*}
\Norm{T_{B_N}(\varphi)^{k} - T_{B_N}(\varphi^{k})}_{\bS_1}
\leq
\Norm{T_{B_N}(\varphi)}\Norm{T_{B_N}(\varphi)^{k-1} - T_{B_N}(\varphi^{k-1})}_{\bS_1} \\
+
\Norm{T_{B_N}(\varphi)T_{B_N}(\varphi^{k-1}) - T_{B_N}(\varphi^{k})}_{\bS_1},
\end{multline*}
it follows from Lemma \ref{lem:approx-product} that if \eqref{eq:approx-STZ2} holds
for some $k-1$, it also holds for $k$. Since it trivially holds for $k=0$, the
claim follows. This immediately gives \eqref{eq:approx-STZ} in the case that
$\varphi$ is a trigonometric polynomial and $f$ is a polynomial.

For the general case, we use an approximation argument. For instance, for
$\varphi\in C(\T)$ and a polynomial $f$, we take $\eps>0$ and choose a 
trigonometric polynomial $\psi$ such that $\Norm{\varphi-\psi}_\infty < \eps$,
then estimate
\begin{align*}
\Norm{T_{B_N}(\beta_N)\SBr{f(T_{B_N}(\varphi)) - T_{B_N}(f\circ\varphi)}}_{\bS_1}
&\leq
\Norm{T_{B_N}(\beta_N)}_{\bS_1}\Norm{f(T_{B_N}(\varphi)) - f(T_{B_N}(\psi))} \\
& \quad +
\Norm{T_{B_N}(\beta_N)\SBr{(T_{B_N}(\psi)) - T_{B_N}(f\circ\psi)}}_{\bS_1} \\
& \quad +
\Norm{T_{B_N}(\beta_N)}_{\bS_1}\Norm{T_{B_N}(f\circ\varphi-f\circ\psi)}.
\end{align*}
We know that for sufficiently large $N$, the second term is smaller than $\eps$.
In addition, by \eqref{eq:Clark-Sp} and \eqref{eq:op-disintegration},
$$
\Norm{T_{B_N}(\beta_N)}_{\bS_1}
\leq
\int_\T \Norm{\beta_N(U^{B_N}_\alpha)} \dm\alpha
=
1.
$$
Then since
$$
\Norm{f(T_{B_N}(\varphi)) - f(T_{B_N}(\psi))}
\leq
C\Norm{T_{B_N}(\varphi) - T_{B_N}(\psi)}
\leq
C\Norm{\varphi - \psi}_\infty
<
C\eps,
$$
and
$$
\Norm{T_{B_N}(f\circ\varphi-f\circ\psi)}
\leq
\Norm{f\circ\varphi - f\circ\psi}_\infty
<
\eps,
$$
the claim follows. Finally, passing from $f$ a polynomial to analytic $f$ is
handled similarly.
\end{proof}

\begin{proof}[Proof of Theorem \ref{thm:angular}.]
If $\Lambda$ is not a Blaschke sequence, then it is known that \ref{angular-1},
\ref{angular-2} and \ref{angular-3} are all satisfied and there is nothing to
prove. So we will suppose that $\Lambda$ is a Blaschke sequence, and let $B$
be the Blaschke product with zero set $\Lambda$.
To ease notation, we will write $\mu_\alpha^N$ instead of $\mu_\alpha^{B_N}$
for the Clark measures associated to $B_N$.

\ref{angular-1}$\Longrightarrow$\ref{angular-2}. Obeserve that
\ref{angular-2} holds precisely if $\beta_N\to0$ almost everywhere on $\T$. Recall
that $\norm{\beta_N(U_\alpha^{B_N})}=\norm{\beta_N}_{L^\infty(\mu_\alpha^N)}$. Then if
\ref{angular-1} holds, we have that
\begin{multline*}
\int_\T \lim_{N\to\infty}\beta_N(\zeta) \dm \zeta
=
\lim_{N\to\infty}\int_\T \beta_N(\zeta) \dm \zeta 
=
\lim_{N\to\infty}\int_\T \int_\T \beta_N(\zeta) \dm\mu_\alpha^N (\zeta) \dm\alpha \\
\leq
\lim_{N\to\infty}\int_\T \norm{\beta_N}_{L^\infty(\mu_\alpha^N)} \dm\alpha
=
0.
\end{multline*}
Since each $\beta_N$ is positive we must have that $\beta_N\to0$ almost everywhere.

\ref{angular-2}$\Longrightarrow$\ref{angular-1}.
Take $\alpha\in\T$ and suppose that
$$
\limsup_{N\to\infty} \norm{\beta_N(U_\alpha^{B_N})}
=
\limsup_{N\to\infty} \norm{\beta_N}_{L^\infty(\mu_\alpha^N)}
\geq \eps.
$$
for some $\eps>0$. Then there is an increasing sequence of integers $N_j$ and
$\zeta_j\in\supp\mu_\alpha^{N_j}$ such that $|B_{N_j}'(\zeta_j)|\leq 1/\eps$.
By passing to a subsequence if necessary we can suppose that $\zeta_j\to\zeta$.

Recall that $I_\delta(\zeta)$ is the arc of $\T$ with centre $\zeta$ and width $2\delta$.
Take $\delta>0$ and $\chi\in C(\T)$ such that $0\leq\chi\leq1$, $\chi\equiv 1$ on
$I_\delta(\zeta)$ and $\supp\chi\subseteq I_{2\delta}(\zeta)$. Then for all sufficiently
large $j$, $\chi(\zeta_j)=1$ and so
$$
\int_\T \chi \dm \mu_\alpha^{N_j}
=\sum_{\eta\in B_{N_j}^{-1}(\alpha)} \frac{\chi(\eta)}{|B_{N_j}'(\eta)|}
\geq
\frac{\chi(\zeta_j)}{|B_{N_j}'(\zeta_j)|}
\geq\eps.
$$
Since $\mu_\alpha^N$ converges in the weak-$^*$ topology to $\mu_\alpha^B$, we conclude
that
$$
\mu_\alpha^B(I_{2\delta}(\zeta)) \geq \int_\T \chi \dm \mu_\alpha^B \geq \eps.
$$
Since this holds for all $\delta>0$, $\mu_\alpha^B$ must have an atom at $\zeta$ and
thus $|B'(\zeta)|<\infty$. In particular, this implies that
$\mu_\alpha^B(\calE)>0$. Consequently, the function
$$
\alpha\mapsto \int_{\calE}  \dm \mu_\alpha^B
$$
is strictly positive on the set
$\CBr{\alpha\in\T:\norm{\beta_N(U_\alpha^{B_N})}\nrightarrow0}$.
So if this set has positive measure,
$$
m(\calE)
=
\int_\T\int_{\calE} \dm \mu_\alpha^B \dm m >0.
$$

\ref{angular-2}$\Longrightarrow$\ref{angular-3}.
Let $\varphi$ be a continuous function. Then, as in the proof of Lemma
\ref{lem:approx}, we compute that
$$
\Tr \Br{T_{B_N}(\beta_N)T_{B_N}(\varphi)}
=
\int_\T \int_\T \beta_N(\zeta) \varphi(\eta) |\wh{k}^{B_N}_\zeta|^2
|B_N'(\zeta)| \dm\zeta\dm\eta
=
\int_\T E_N\varphi(\zeta)\dm\zeta.
$$
By Lemma \ref{lem:fejer}\ref{lem:fejer-3}, we have that 
$E_n\varphi\to\varphi$ almost everywhere on $\T\setminus\calE$, and Since
$\Norm{E_N\varphi}_\infty \leq\Norm{\varphi}_\infty$ for each $N$,
\[
\int_{\T\setminus\calE} E_N\varphi(\zeta)\dm\zeta
\to
\int_{\T\setminus\calE} \varphi(\zeta)\dm\zeta
\quad\text{as}\quad N\to\infty.
\label{eq:integral}
\]
This shows that if $m(\calE)=0$, then
\[
\lim_{N\to\infty}\Tr \Br{T_{B_N}(\beta_N)T_{B_N}(\varphi)}
=
\int_\T \varphi(\zeta)\dm \zeta.
\label{eq:STZ-reduced}
\]
We conclude by by noting that if $f$ is analytic on a neighbourhood of
$\conv\calR(\varphi)$, then
\begin{align*}
& \Abs{\Tr \SBr{T_{B_N}(\beta_N)f(T_{B_N}(\varphi))}
-
\int_\T f(\varphi(\zeta))\dm \zeta} \\
&\leq
\Abs{\Tr \SBr{T_{B_N}(\beta_N)T_{B_N}(f\circ\varphi))} 
-
\int_\T f(\varphi(\zeta))\dm \zeta} \\
& \quad +
\Norm{T_{B_N}(\beta_N)\SBr{f(T_{B_N}(\varphi)) - T_{B_N}(f\circ\varphi)}}_{\bS_1}.
\end{align*}
Using \eqref{eq:STZ-reduced} with $f\circ\varphi$ in place of $\varphi$ we see that
that first term on the right goes to zero, and Lemma \ref{lem:approx-STZ} shows the
second term goes to zero.

\ref{angular-3}$\Longrightarrow$\ref{angular-2}.
To begin with, we observe that for $\zeta\in\calE$ and $N\in\N$,
$k^{B_N}_\zeta=P_{B_N}k^{B}_\zeta$, where $P_{B_N}$ is the orthogonal
projection onto $K_{B_N}$. Indeed, for each $f\in K_B$,
$$
\Br{f,k^{B_N}_\zeta}
=
P_{B_N}f(\zeta)
=
\Br{P_{B_N}f, k^{B}_\zeta}
=
\Br{f, P_{B_N}k^{B}_\zeta}.
$$
It follows that $k^{B_N}_\zeta\to k^{B}_\zeta$ in $L^2(\T)$. Consequently,
$\wh{k}^{B_N}_\zeta\to \wh{k}^{B}_\zeta$ in $L^2(\T)$, and
then the Dominated Convergence Theorem implies that for each $\varphi\in C(\T)$,
$$
\lim_{N\to\infty}\int_{\calE} E_N\varphi(\zeta)\dm\zeta
=
\int_\calE \int_\T \varphi(\eta) |\wh{k}_\zeta^{B} (\eta)|^2 \dm\eta \dm\zeta
=
\int_\T \varphi(\eta) \int_\calE |\wh{k}_\zeta^{B} (\eta)|^2 \dm\zeta \dm\eta. 
$$
Combining this with \eqref{eq:integral} we see that \eqref{eq:STZ-reduced} holds for
all $\varphi\in C(\T)$ if and only if the function
$$
h(\eta) = \int_\calE |\wh{k}_\zeta^{B} (\eta)|^2 \dm\zeta
$$
is equal almost everywhere to the characteristic function of $\calE$. We will
show that this cannot be true if $m(\calE)>0$. We consider two cases separately.

\textbf{Case 1:} $0<m(\calE)<1$.

For $R\geq1$, set
\begin{align*}
\calE_0 &= \calE_0(R) = \CBr{\zeta\in\T :\, |B'(\zeta)|\leq R}, \\
G_0 &= G_0(R) = \CBr{\alpha\in\T :\, \mu_\alpha^B(\calE_0(R))\geq 1/R}.
\end{align*}
Observe that since the sets $G_0(R)$ are increasing with $R$, and
$$
m\Br{\bigcup_{R\geq1}^\infty G_0(R)}
=
m\Br{\CBr{\alpha\in\T :\, \mu_\alpha^B(\calE)>0}} > 0,
$$
$m(G_0(R))>0$ for all sufficiently large $R$. From now on we fix such an $R$
and set $\delta = m(G_0)/3$. Then for each $\eta\in\T$,
\begin{multline*}
h(\eta)
\geq
\int_{\calE_0} |\wh{k}_\zeta^{B} (\eta)|^2 \dm\zeta
=
\int_\T \int_{\calE_0} |\wh{k}_\zeta^{B} (\eta)|^2
\dm\mu_\alpha^B(\zeta) \dm\alpha \\
\geq
\int_{G_0\setminus I_\delta(B(\eta))} \int_{\calE_0}
|\wh{k}_\zeta^{B} (\eta)|^2 \dm\mu_\alpha^B(\zeta) \dm\alpha.
\end{multline*}
Note that $m(G_0\setminus I_\delta(B(\eta)))\geq \delta$. Moreover, if
$\alpha\in\T\setminus I_\delta(B(\eta))$ and $\zeta\in\calE_0\cap B^{-1}(\alpha)$,
then
$$
|\wh{k}_\zeta^{B} (\eta)|^2
=
\frac{1}{|B'(\zeta)|}\Abs{\frac{B(\zeta)-B(\eta)}{\zeta - \eta}}^2
\geq
\frac{\delta^2}{4|B'(\zeta)|}
\geq
 \frac{\delta^2}{4R}.
$$
Combining the previous two estimates we conclude that
$$
h(\eta)
\geq
\frac{\delta^2}{4R}
\int_{G_0\setminus I_\delta(B(\eta))} \mu_\alpha^B(\calE_0) \dm\alpha
\geq
\frac{\delta^3}{4R^2}
>0.
$$
Since $h$ is bounded away from zero, $h$ cannot be equal to the characteristic
function of $\calE$.

\textbf{Case 2:} $m(\calE)=1$.

Set $M=\ess\-\inf|B'|$. Recall that since $B(0)=0$, $M\geq1$. For $\eps>0$, set
$V=\CBr{\zeta\in\T: \, |B'(\zeta)|\leq M+\eps}$. Note that $m(V)>0$ for all
$\eps>0$. We will show that for sufficiently small $\eps$, $h$ is bounded away
from $1$ on $V$.

For each $\eta\in V$, we have that
\begin{align*}
h(\eta) 
&=
\int_\T \frac{|k_\eta^{B} (\zeta)|^2}{|B'(\zeta)|} \dm\zeta \\
&=
1 - \int_\T |k_\eta^{B} (\zeta)|^2
\Br{\frac{1}{|B'(\eta)|}-\frac{1}{|B'(\zeta)|}} \dm\zeta \\
&\leq
1 - \int_\T |k_\eta^{B} (\zeta)|^2
\Br{\frac{1}{M+\eps}-\frac{1}{|B'(\zeta)|}} \dm\zeta \\
&\leq
1 - \int_\T |k_\eta^{B} (\zeta)|^2
\Br{\frac{1}{M}-\frac{1}{|B'(\zeta)|}} \dm\zeta + \frac{\eps}{M}.
\end{align*}
Let us suppose that
\[
A
:=
\inf_{\eta\in\T} \int_\T |k_\eta^{B} (\zeta)|^2
\Br{\frac{1}{M}-\frac{1}{|B'(\zeta)|}} \dm\zeta
> 0.
\label{eq:STZ-integral}
\]
Then by taking $0<\eps\leq A/2$, we get that $h(\eta)\leq 1 -\eps$ for all
$\eta\in V$. So it only remains to prove \eqref{eq:STZ-integral}. We will use an
argument similar to that in Case 1 above.

For large $R>0$, set
\begin{align*}
\calE_1 &= \CBr{\zeta\in\T :\, M+1/R \leq |B'(\zeta)| \leq R}, \\
G_1 &= \CBr{\alpha\in\T :\, \mu_\alpha^B(\calE_1)>0}.
\end{align*}
As before, we have that $m(\calE_1)>0$ for all sufficiently large $R$ and we choose
$R$ such that this holds. Set $\delta=m(G_1)/3$ and take $\eta\in\T$. Then for $\alpha\in
G_1\setminus I_{\delta}(B(\eta))$, $\mu_\alpha^B$ must have an atom at some point
$\zeta_0\in\calE_1$. Moreover, $|k_\eta^{B} (\zeta_0)|\geq \delta/2$, and so it follows
that in this case
$$
\int_\T |k_\eta^{B} (\zeta)|^2
\Br{\frac{1}{M}-\frac{1}{|B'(\zeta)|}} \dm\mu_\alpha^B(\zeta)
\geq
\frac{|k_\eta^{B} (\zeta_0)|^2}{|B'(\zeta_0)|}\Br{\frac{1}{M}-\frac{1}{|B'(\zeta_0)|}}
\geq
\frac{\delta^2}{8M^2R^2}.
$$
Integrating with respect to $\alpha$, we conclude that
$$
\int_\T |k_\eta^{B} (\zeta)|^2
\Br{\frac{1}{M}-\frac{1}{|B'(\zeta)|}} \dm\zeta
\geq
\int_{G_1\setminus I_{\delta}(B(\eta))} 
\frac{\delta^2}{8M^2R^2} \dm\alpha
\geq
\frac{\delta^3}{8M^2R^2}.
$$
\end{proof}

%%%%%%%%%%%%%%%%%%%%%%%%%%%%%%%%%%%%%%%%%%%%%%%%%%%%%%%%%%%%%%%%%%%%%%%%%%%


\begin{thebibliography}{99}

\bibitem{Bottcher1}
{\sc A.~B\"{o}ttcher,}
\emph{Borodin–Okounkov and Szeg\H{o} for Toeplitz Operators on Model Spaces,}
Integral Equations Operator Theory \textbf{78} (2014), 407--414.

\bibitem{BotSil1}
{\sc A.~B\"{o}ttcher, B.~Silberman}
\emph{ Introduction to large truncated Toeplitz matrices,}
Universitext. Springer-Verlag, New York, 1999.

\bibitem{GMR1}
{\sc S.~Garcia, J.~Mashreghi, W.~Ross,}
\emph{Introduction to Model Spaces and their Operators,}
Cambridge University Press, Cambridge, 2016.

\bibitem{Gray1}
{\sc R.M.~Gray,}
\emph{Toeplitz and Circulant Matrices: A Review,}
Foundations and Trends in Communications and Information Theory,
Vol. 2 (2006): no.~3, 155–239.

\bibitem{Matheson-Stessin1}
{\sc A.~L.~Mathesson, M.~I.~Stessin,}
\emph{Applications of spectral measures,}
in Recent Advances in Operator-Related Function Theory,
\textbf{393}, Comtemp. Math, 15--27, Providence, RI,
Amer. Math. Soc., 2006,


\bibitem{Mashregi1}
{\sc J.~Mashregi,}
\emph{Derivatives of Inner Functions,}
Springer, New York, 2014.

\bibitem{Nikolski1}
{\sc N.~Nikolski,}
\emph{Toeplitz Matrices and Operators,}
Cambridge University Press, Cambridge, 2020.

\bibitem{Saksman1}
{\sc E.~Saksman,}
\emph{An elemenrary introduction to Clark measures,}
in Topics in Complex Analysis and Operator Theory,
85--136", Universidad de M\'{a}laga, M\'{a}laga,
2007.

\bibitem{STZ1}
{\sc E.~Strouse, D.~Timotin, M.~Zarrabi,}
\emph{A Szeg\H{o} type theorem for truncated Toeplitz operators,}
J. Approx. Theory \textbf{220} (2017), 12--30.

\end{thebibliography}
\end{document}